\documentclass[letterpaper, 10 pt, conference]{ieeeconf}  

\IEEEoverridecommandlockouts                              

\usepackage{todonotes}
\usepackage{latexsym,amssymb,amsmath, amsbsy,amsopn, amstext}
\usepackage{graphicx,epsfig,tikz,pgf,hyperref,cancel,color,float,ifpdf}
\usepackage{extarrows,mathrsfs}
\usepackage{amsmath,  amssymb,stfloats}
\usepackage{subcaption}
\usepackage{tikz}
\usepackage{textgreek}
\usetikzlibrary{automata, positioning, arrows}
\usepackage[ruled,vlined,linesnumbered]{algorithm2e}
\usepackage{multirow,xurl}
\usepackage{bm}

\usepackage{etoolbox}
\makeatletter
\patchcmd{\@makecaption}
  {\scshape}
  {}
  {}
  {}
\makeatother
\usepackage{makecell}

\makeatletter
\patchcmd{\@algocf@start}
  {-1.5em}
  {0pt}
  {}{}
\makeatother

\graphicspath{{./figures/}}

\newtheorem{definition}{\bfseries Definition}
\newtheorem{proposition}{\bfseries Proposition}
\newtheorem{theorem}{\bfseries Theorem}
\newtheorem{corollary}{\bfseries Corollary}
\newtheorem{lemma}{\bfseries Lemma}

\DeclareMathOperator{\itl}{interval}

\DeclareMathOperator{\proj}{proj}

\def\x{\bm{x}}

\def\y{\bm{y}}

\newcommand{\mt}[1]{\boldsymbol{#1}}

\newcommand{\li}{[\![}
\newcommand{\ri}{]\!]}

\newcommand{\cc}{\mathbf{c}}
\newcommand{\G}{\mathbf{G}}
\newcommand{\A}{\mathbf{A}}
\newcommand{\bb}{\mathbf{b}}

\newcommand{\Zc}{\mathcal{Z}}

\newcommand{\mat}[1]{\begin{bmatrix} #1 \end{bmatrix}}
\newcommand{\hz}[1]{\langle \G^c_{#1},\allowbreak \G^b_{#1},\allowbreak \cc_{#1},\allowbreak \A^c_{#1},\allowbreak \A^b_{#1},\allowbreak \bb_{#1} \rangle}

\newcommand{\lrangle}[1]{\langle #1 \rangle}
\newcommand{\fl}[1]{\lfloor #1 \rfloor}
\allowdisplaybreaks[4]

\usepackage{xcolor}
\newif\ifdraft

\drafttrue

\title{\LARGE \bf Efficient Reachability Analysis for Convolutional Neural Networks Using Hybrid Zonotopes}

\author{Yuhao Zhang and Xiangru Xu
\thanks{Y. Zhang and X. Xu are with the Department of Mechanical Engineering, University of Wisconsin-Madison, Madison, WI, USA. Email:         {\tt\small \{yuhao.zhang2,xiangru.xu\}@wisc.edu}.}%
}

\begin{document}
\maketitle
\begin{abstract}
Feedforward neural networks are widely used in autonomous systems, particularly for control and perception tasks within the system loop. However, their vulnerability to adversarial attacks necessitates formal verification before deployment in safety-critical applications. Existing set propagation-based reachability analysis methods for feedforward neural networks often struggle to achieve both scalability and accuracy. This work presents a novel set-based approach for computing the reachable sets of convolutional neural networks. The proposed method leverages a hybrid zonotope representation and an efficient neural network reduction technique, providing a flexible trade-off between computational complexity and approximation accuracy. Numerical examples are presented to demonstrate the effectiveness of the proposed approach.
\end{abstract}

\section{Introduction}\label{sec:intro}

The integration of neural networks into autonomous systems has revolutionized how these systems interact with and respond to their environments. For instance, Fully-connected Feedforward Neural Networks (FFNNs) are commonly used to approximate nonlinear and complex control policies, while Convolutional Neural Networks (CNNs), with their powerful ability to process and interpret intricate visual data, play a critical role in perception tasks. However, deploying neural networks in safety-critical applications introduces significant challenges, particularly due to their vulnerability to adversarial attacks, where minor input perturbations can lead to substantial deviations in output \cite{goodfellow2014explaining,su2019one}.

Recent advancements have focused on reachability-based methods for the formal verification of deep neural networks, leveraging their computational efficiency. Many of these methods focus on analyzing FFNNs by abstracting the non-linear activation functions using various set representations, enabling both forward and backward reachability analysis through set-propagation techniques \cite{everett2021reachability, zhang22constrainedzonotope,rober2023backward,zhang2023backward}. Despite their theoretical soundness, these approaches often face challenges with scalability and flexibility, especially when applied to large-scale FFNNs. Compared to FNNs, CNNs possess more complex architectures, incorporating convolutional layers and max pooling layers. The diversity of layer types and the depth of these architectures present significant challenges in developing efficient and robust verification methods. Although several methods exist for handling CNNs such as \cite{zhang2018efficient,yang2021reachability,katz2019marabou, tran2020verification, santa2022nnlander, katz2022verification, hsieh2022verifying}, only a few are capable of handling real-world networks while providing the flexibility needed for efficient and accurate analysis.

Recently, an approach based on \emph{Hybrid Zonotopes} (HZs) was proposed for the reachability analysis and verification of FFNNs \cite{zhang2023backward, zhang2023reachability, zhang2024reachability,zhang2024hybrid}. With the capability of representing non-convex sets with flat faces \cite{bird2023hybrid}, HZs enable exact abstractions of ReLU-activated FFNNs through simple matrix operations. Additionally, HZs provide a unified tool for handling both non-convex sets and nonlinear plant models. Nevertheless, the aforementioned approaches are limited to FFNNs and also face scalability challenges. 

To address these challenges, this work introduces an efficient set-based method for computing the reachable sets of FFNNs and CNNs. The proposed approach combines HZ representation with an efficient neural network reduction technique, providing the flexibility to balance computational complexity and approximation accuracy. The contributions of this work are twofold: i) A flexible and efficient method is proposed for reducing the number of neurons in a FFNN while preserving its intrinsic input-output mapping properties, with bounded approximation errors. ii) Based on this neuron network reduction, an HZ-based approach is presented for computing the reachable sets of FFNNs and CNNs by transforming various CNN layers into equivalent fully-connected layers. The constructed HZ representations are shown to over-approximate the exact reachable sets, while exact reachability analysis can be recovered as a special case. The performance of the proposed method is demonstrated through numerical examples.

\section{Preliminaries \& Problem Statement}\label{sec:pre}

\emph{Notation.} 
The $i$-th component of a vector $\x\in \mathbb{R}^n$ is denoted by $x_{i}$ with $i\in [n]\triangleq \{1,\dots,n\}$. The entry in the $i$-th row and $j$-th columm of a matrix $\A\in \mathbb{R}^{n\times m}$ is denoted by $\A_{[i,j]}$. The $i$-th row (resp. $j$-th columm) of $\A$ is denoted by $\A_{[i,:]}$ (resp. $\A_{[:,j]}$). For a set $\mathcal{C}\subset[n]$ (resp. $\mathcal{C}'\subset[m]$), $\A_{[\mathcal{C},:]}$ (resp. $\A_{[:,\mathcal{C}']}$) denotes a submatrix of $\A$ with all rows $i\in\mathcal{C}$ (resp. all columns $j\in\mathcal{C}'$). 
Given sets $\mathcal{X}\subset \mathbb{R}^n$, $\mathcal{Z}\subset \mathbb{R}^m$ and a matrix $\bm{R}\in\mathbb{R}^{m\times n}$, 
the generalized intersection of $\mathcal{X}$ and $\mathcal{Z}$ under $\bm R$ is $\mathcal{X} \cap_{\bm R}\mathcal{Z} = \{\x\in\mathcal{X}\;|\;\bm R \x\in\mathcal{Z}\}$. 
An interval with bounds $\underline{\mt{a}}$, $\overline{\mt{a}} \in \mathbb{R}^n$ is denoted as $\li \underline{\mt{a}}, \overline{\mt{a}}\ri$. 
The projection of a set $\mathcal{X}\subset \mathbb{R}^n$ onto a set of coordinates $\Phi = \{i_1,\dots,i_k\} \subset [n]$ is denoted as $\proj_\Phi(\mathcal{X}) \triangleq \{[\bm{e}_{i_1}\; \;\cdots\; \; \bm{e}_{i_k}]^\top \bm x\; |\; \bm x \in \mathcal{X}\} \subset \mathbb{R}^k$. 
The cardinality of the set $\Phi$ is denoted as $|\Phi|$. 
Given a vector $\x\in \mathbb{R}^n$, the maxout function is $\max(\x) \triangleq \max_{i\in[n]} x_i$.

First, we provide the definition of hybrid zonotope. 
\begin{definition}\cite[Definition 3]{bird2023hybrid}\label{def:sets}
The set $\Zc \subset\mathbb{R}^n$ is a hybrid zonotope if there exist $\mathbf{c} \in \mathbb{R}^{n}$, $\mathbf{G}^c \in \mathbb{R}^{n \times n_{g}}$, $\G^b\in \mathbb{R}^{n \times n_{b}}$, $\A^c \in \mathbb{R}^{n_{c}\times{n_g}}$, $\A^b \in\mathbb{R}^{n_{c}\times{n_b}}$, $\bb \in \mathbb{R}^{n_{c}}$ such that $\mathcal{Z}=\{ \G^c  \bm{\xi}^c+\G^b \bm{\xi}^b+\cc \mid 
\bm{\xi}^c\in \mathcal{B}_{\infty}^{n_{g}}, \bm{\xi}^b\in\{-1,1\}^{n_{b}}, 
\A^c\bm{\xi}^c+\A^b \bm{\xi}^b=\bb
\}$ where $\mathcal{B}_{\infty}^{n_g}=\left\{\bm{x} \in \mathbb{R}^{n_g} \;|\;\|\bm x\|_{\infty} \leq 1\right\}$ is the unit hypercube in $\mathbb{R}^{n_{g}}$. The \emph{HCG}-representation of $\Zc$  is given by $\mathcal{Z}= \hz{}$.
\end{definition}

Next, we introduce the notations related to FFNNs, which consist exclusively of fully-connected layers. Let $\mt{\pi}_{F}:\mathbb{R}^n \rightarrow \mathbb{R}^m$ be an $\ell$-layer FFNN with weight matrices $\{\bm W^{(k)}\}_{k\in[\ell]}$ and bias vectors $\{\bm v^{(k)}\}_{k\in[\ell]}$. Denote $\bm{x} \in \mathbb{R}^{n}$ as the input of $\mt\pi_{F}$ and $\bm x^{(k)}\in \mathbb{R}^{n_k}$ as the neurons of the $k$-th layer. Then, 
\begin{subequations}\label{equ:FFNN}
\begin{align}
    \bm x^{(0)} & = \bm{x},\\
    \bm x^{(k)} & = \mathcal{L}_{fc}(\bm W^{(k)},\bm v^{(k)}, \bm x^{(k-1)}), \text{ for } k \in [\ell-1],\\
    \mt{\pi}_{F}(\bm{x}) & = \bm W^{(\ell)}\bm x^{(\ell-1)}+ \bm v^{(\ell)},
\end{align}    
\end{subequations}
where $$
\mathcal{L}_{fc}(\bm W^{(k)},\bm v^{(k)}, \bm x^{(k-1)}) \triangleq \mt{\sigma}(\bm W^{(k)}\bm x^{(k-1)}+ \bm v^{(k)})
$$   
and $\mt{\sigma}$ is the vector-valued activation function constructed by component-wise repetition of the activation function $\sigma(\cdot)$, i.e., $\mt{\sigma}(\bm z) \triangleq [\sigma(z_1) \;\cdots\; \sigma(z_{n_k})]^\top.$ In this work, $\sigma$ is assumed to be the ReLU function, but the proposed method can be easily extended to other activation functions (such as sigmoid and
tanh) by employing their HZ approximations  \cite{zhang2024hybrid}. 

Given an input set $\mathcal{Z} \subset\mathbb{R}^n$ of the FFNN $\mt{\pi}_F$, the (output) \emph{reachable set} of $\mathcal{Z}$ is defined as 
$$
\mathcal{R}_{\mt{\pi}_F}(\mathcal{Z}) \triangleq \{\bm{z}\in \mathbb{R}^{m}\;|\; \bm z = \mt{\pi}_F(\bm x), \bm x \in\mathcal{Z}\}$$ 
and the \emph{graph set} of $\mt{\pi}_F$ over $\mathcal{Z}$ is defined as 
$$\mathcal{G}_{\mt{\pi}_F}(\mathcal{Z}) \triangleq \{(\bm{x},\bm{z})\in \mathbb{R}^{n+m}\;|\; \bm z = \mt{\pi}_F(\bm x), \bm x \in\mathcal{Z}\}.
$$ 

Now we introduce notations related to CNNs, which combine fully connected layers with non-fully-connected layers. Let $\mt{\pi}_C:\mathbb{R}^{c_I \times h_I\times w_I} \rightarrow \mathbb{R}^m$ be an $\ell$-layer CNN with parameters $\Theta \triangleq (\bm\theta^{(1)}, \bm\theta^{(2)}, \dots, \bm\theta^{(\ell)})$, where $\bm\theta^{(k)}$ denote the learnable parameters and hyperparameters associated with the $k$-th layer, for $k\in [\ell]$. We assume $\mt{\pi}_C$ consists of the following commonly-used layer types: \emph{convolution layer} $\mathcal{L}_{conv}$, \emph{average pooling layer} $\mathcal{L}_{ap}$, \emph{max pooling layer} $\mathcal{L}_{mp}$, and \emph{fully-connected layer} $\mathcal{L}_{fc}$. 
Given an image input to the CNN $\mt{\pi}_C$ as $I\in\mathbb{R}^{c_I \times h_I\times w_I}$ with width $w_I$, height $h_I$ and number of channels $c_I$, the output of $\mt{\pi}_C$ is computed as 
\begin{subequations}\label{equ:CNN}
\begin{align}
 \bm{x}^{(0)} & = I, \\
 \bm{x}^{(k)} & = \mathcal{L}^{(k)}(\bm{\theta}^{(k)}, \bm{x}^{(k-1)}), \text{ for } k \in [\ell],\\
\mt{\pi}_C(I) & = \bm{x}^{(\ell)},
\end{align}
\end{subequations}
where $\mathcal{L}^{(k)} \in \{\mathcal{L}_{conv}, \mathcal{L}_{ap}, \mathcal{L}_{mp}, \mathcal{L}_{fc}\}$, for $k \in [\ell]$. Depending on the type of the layer, the neurons of the $k$-th layer $\x^{(k)}$ is either a 3D tensor or a vector, i.e., $\x^{(k)} \in \mathbb{R}^{c_k\times h_k \times w_k}$ for $\mathcal{L}^{(k)} \in \{\mathcal{L}_{conv}, \mathcal{L}_{ap}, \mathcal{L}_{mp}\}$ and $\x^{(k)} \in \mathbb{R}^{n_k}$ for $\mathcal{L}^{(k)} = \mathcal{L}_{fc}$. 
The associated parameters $\bm{\theta}^{(k)}$ may include learnable parameters such as weights $\bm{W}^{(k)}$ and bias $\bm{v}^{(k)}$, as well as hyperparameters such as padding sizes $(p_h^{(k)}, p_w^{(k)})$ and strides $(s_h^{(k)},s_w^{(k)})$. 

Given an input set $\mathcal{I} \subset \mathbb{R}^{c_I \times h_I\times w_I}$ to $\mt{\pi}_C$, which can be considered as a set of perturbed images, the reachable set of $\mt{\pi}_C$ that contains all possible outputs is defined as 
$$
\mathcal{R}_{\mt{\pi}_C}(\mathcal{I}) \triangleq \{\bm{y}\in \mathbb{R}^{m}\;|\; \bm y = \mt{\pi}_C(\bm x), \bm x \in\mathcal{I}\}
$$ 
and the graph set of $\mt{\pi}_C$ over $\mathcal{I}$ is defined as 
$$
\mathcal{G}_{\mt{\pi}_C}(\mathcal{I}) \triangleq \{(\bm{x},\bm{y})\in \mathbb{R}^{c_I \times h_I\times w_I+m}\;|\; \bm y = \mt{\pi}_C(\bm x), \bm x \in\mathcal{I}\}.$$ 

In this work, we aim to develop an efficient HZ-based approach for computing an over-approximated reachable set, $\hat{\mathcal{R}}_{{\mt{\pi}}_C}(\mathcal{I})$, for a given CNN $\mt{\pi}_C$ and an input set $\mathcal{I}$, such that $\hat{\mathcal{R}}_{{\mt{\pi}}_C}(\mathcal{I}) \supseteq \mathcal{R}_{\mt{\pi}_C}(\mathcal{I})$. This method provides the flexibility to balance between computational efficiency and approximation accuracy. To that end, we first transform the CNN $\mt{\pi}_C$ into an equivalent FFNN ${\mt{\pi}}_F$ (see Section \ref{subsec:CNN}), and then develop a flexible NN reduction method for ${\mt{\pi}}_F$ that enables the computation of a closed-form over-approximation $\hat{\mathcal{R}}_{{\mt{\pi}}_F}(\mathcal{I})$, which is equivalent to $\hat{\mathcal{R}}_{{\mt{\pi}}_C}(\mathcal{I})$   (see Section \ref{sec:FFNN}).

\section{Transformation From CNN Into FFNN }\label{subsec:CNN}

In this section, we convert  each of the convolution and pooling layers  in $\mt{\pi}_C$ into a fully connected layer, such that a given CNN $\mt{\pi}_C:\mathbb{R}^{c_I \times h_I\times w_I} \rightarrow \mathbb{R}^m$ is transformed into an equivalent FFNN ${\mt{\pi}}_F: \mathbb{R}^{c_I \cdot h_I \cdot w_I} \rightarrow \mathbb{R}^m$. 
The definitions of convolution layers and pooling layers are based on \cite[Chapter 9]{Goodfellow2016deep}.

\subsection{Convolution Layer}

The convolution layer of a CNN is defined below.

\begin{definition}\label{def:conv}
Given an input $\x \in \mathbb{R}^{c_x \times h_x \times w_x}$ to the convolution layer $\mathcal{L}_{conv}$ of a CNN with activation function $\sigma: \mathbb{R} \rightarrow \mathbb{R}$ and parameters $\bm{\theta}_{conv} = (\bm{W}, \bm{v}, p_h, p_w, s_h, s_w)$, where $\bm{W} \in \mathbb{R}^{f_W \times c_x \times h_W \times w_W}$ are the filter weights with $f_W$ denoting the number of filters, $\bm{v} \in \mathbb{R}^{f_W \times 1 \times 1}$ is the bias, $p_w \in \mathbb{R}$ and $p_h \in \mathbb{R}$ are the padding sizes along the width and height, respectively, and $s_w \in \mathbb{R}$ and $s_h \in \mathbb{R}$ are the strides along the width and height, the output of the convolution layer, $\bm{y} = \mathcal{L}_{conv} (\bm{\theta}_{conv}, \x) \in \mathbb{R}^{c_y \times h_y \times w_y}$, is computed as follows:
    \begin{equation}\label{equ:conv-layer}
        \begin{aligned}
            y_{[i_{c_y},i_{h_y},i_{w_y}]}  =  \sigma( & \sum_{i_{c_x} = 1}^{c_x}\sum_{i_{h_W} = 1}^{h_W} \sum_{i_{w_W} = 1}^{w_W} v_{[i_{c_y},1,1]} \;+ \\ & W_{[i_{c_y},i_{c_x},i_{h_W},i_{w_W}]}\cdot \bar x_{[i_{c_x}, j_h, j_w]}),
        \end{aligned}
    \end{equation}
    where $i_{c_y} \in [c_y], i_{h_y} \in [h_y], i_{w_y} \in [w_y]$,  
    \begin{align*}
        j_h & = i_{h_W}+(i_{h_y}-1)s_h-p_h,\\
        j_w &= i_{w_W}+(i_{w_y}-1)s_w-p_w,
    \end{align*}
    and 
    \begin{align*}
        \bar x_{[i_{c_x}, j_h, j_w]} = \left\{ \begin{array}{ll}
x_{[i_{c_x}, j_h, j_w]}, & \mbox{if } ( j_h \in [h_x]) \wedge (  j_w \in [w_x]),\\ 
0, & \mbox{otherwise.}
\end{array}\right.
    \end{align*}
    The number of output channels is $c_y = f_W$, while the height and width of the output are $h_y = \fl{\frac{h_x-h_W+p_h+s_h}{s_h}}$ and $w_y = \fl{\frac{w_x-w_W+p_w+s_w}{s_w}}$, respectively.
\end{definition}

From the definition above, it can be observed that the convolution layer consists of a linear mapping of the tensor input $\x$, weights $\bm W$,  bias $ \bm v$, and a nonlinear activation function $\sigma$. 
The following lemma shows that the convolution layer can be transformed into an equivalent fully-connected layer by flattening the input $\x$ and output $\bm y$ into vectors following the row-major order:
\begin{subequations}\label{equ:flatten}
\begin{align}
    \tilde{\x}  & \triangleq \mat{  x_{[1,1,1]} & x_{[1,1,2]} & \cdots & x_{[c_x,h_x,w_x]} }^\top, \\ 
    \tilde{\y}  & \triangleq \mat{  y_{[1,1,1]} & y_{[1,1,2]} & \cdots & y_{[c_y,h_y,w_y]} }^\top.
\end{align}
\end{subequations}

\begin{lemma}\label{lem:conv}
Given a convolution layer $\mathcal{L}_{conv}$ with activation function $\sigma: \mathbb{R} \rightarrow \mathbb{R}$ and parameters $\bm{\theta}_{conv} = (\bm{W}, \bm{v}, p_h, p_w, s_h, s_w)$, let ${\x}\in \mathbb{R}^{c_x \times h_x \times w_x}$ and ${\y}\in \mathbb{R}^{c_y \times h_y \times w_y}$ represent the input and the corresponding output of  $\mathcal{L}_{conv}$, and let $\tilde{\x}\in \mathbb{R}^{c_x \cdot h_x \cdot w_x}$ and $\tilde{\y}\in \mathbb{R}^{c_y \cdot h_y \cdot w_y}$ denote the flattened input and output vectors defined in \eqref{equ:flatten}, respectively. Then,
\begin{align}
\tilde{\y} = \bm \sigma(\tilde {\bm W}\cdot \tilde{\x} + \tilde {\bm v})
\end{align}
where 
\begin{align*}
\tilde {\bm W} &= \tilde {\bm W}^{conv}\cdot \tilde {\bm W}^p,\\
\tilde {\bm v} &= [\bm v_{[1,1,1]}\;\cdots\;\bm v_{[c_y,1,1]}]^\top \otimes \bm 1_{h_y\cdot w_y \times 1},    
\end{align*}
$\bm \sigma$ is the vector-valued activation function, $\otimes$ is the Kronecker product, and 
matrices $\tilde {\bm W}^p \in \mathbb{R}^{(c_x \cdot (h_x+2p_h) \cdot (w_x+2p_w)) \times (c_x \cdot h_x \cdot w_x)}$ and $\tilde {\bm W}^{conv}\in \mathbb{R}^{ (c_y \cdot h_y \cdot w_y)\times (c_x \cdot (h_x+2p_h) \cdot (w_x+2p_w))}$ are zero matrices except for the following entries:
\begin{align*}
& \tilde {\bm W}^p_{[\mathfrak i_{row},\mathfrak i_{col}]} = 1,\\
& \tilde {\bm W}^{conv}_{[\mathfrak j_{row},\mathfrak j_{col}]} = \bm{W}_{[i_{c_y},i_{c_x},i_{h_W},i_{w_W}]},
\end{align*}
with 
\begin{align*}
\mathfrak i_{row} &=   (i_{c_x}\!-\!1)\!\cdot\!(h_x\!+\!2p_h) \!\cdot\! (w_x\!+\!2p_w) \!+\! p_h\!\cdot\!(w_x\!+\!2p_w) \\ &\!+\! (i_{h_x}\!-\!1)\!\cdot\!(w_x\!+\!2p_w) \!+\! p_w \!+\! i_{w_x},\\
\mathfrak i_{col} &=  (i_{c_x}\!-\!1)\!\cdot\! h_x\!\cdot\! w_x \!+\!  (i_{h_x}\!-\!1)\!\cdot\! w_x \!+\! i_{w_x},\\
\mathfrak j_{row} &=  (i_{c_y}\!-\!1) \!\cdot\! h_y \!\cdot\!w_y \!+\! (i_{h_y}\!-\!1)\!\cdot\! w_y \!+\!i_{w_y},\\
\mathfrak j_{col} &=  (i_{c_x}\!-\!1)\!\cdot\!(h_x\!+\!2p_h) \!\cdot\! (w_x\!+\!2p_w) \!+\! (i_{w_y}\!-\!1) \!\cdot\! s_w \!+\! i_{w_W} \\ & \!+\! (i_{h_W}\!-\!1)\!\cdot\! (w_x\!+\!2p_w) \!+\! (i_{h_y}\!-\!1)\!\cdot\! (w_x\!+\!2p_w)\!\cdot\! s_h ,
\end{align*}
for any $i_{c_x}\in[{c_x}],  i_{h_x}\in[{h_x}], i_{w_x}\in[{w_x}], i_{c_y}\in[{c_y}],  i_{h_y}\in[{h_y}], i_{w_y}\in[{w_y}], i_{h_W}\in[{h_W}], i_{w_W}\in[{w_W}].$
\end{lemma}

Using Lemma \ref{lem:conv}, the convolution layers in a CNN can be transformed into fully-connected layers with flattened input $\tilde x$, output $\tilde y$, weight matrix $\tilde {\bm W}$ and bias $\tilde {\bm v}$. Although $\tilde {\bm W}$ is a large matrix, the majority of its entries are zeros, allowing it to be stored in a sparse matrix format to reduce the computational overhead.

\subsection{Average Pooling Layer}
The average pooling layer of a CNN is defined as follows.
\begin{definition}\label{def:avg-pool}
Given an input $\x \in \mathbb{R}^{c_x \times h_x \times w_x}$ to the average pooling layer $\mathcal{L}_{ap}$ with parameters $\bm{\theta}_{ap} = (h_{p}, w_{p}, p_h, p_w, s_h, s_w)$, where $h_{p}\in \mathbb{R}$ and $w_{p}\in \mathbb{R}$ are the pooling window sizes, $p_w \in \mathbb{R}$ and $p_h \in \mathbb{R}$ are the padding sizes, and $s_w \in \mathbb{R}$ and $s_h \in \mathbb{R}$ are the strides along the width and height, respectively, the output of the averaging pooling layer,  $\bm{y} = \mathcal{L}_{ap} (\bm{\theta}_{ap}, \x) \in \mathbb{R}^{c_x \times h_y \times w_y}$, is computed as follows:
\begin{align}\label{equ:avg-pool}
        \y_{[i_{c_x},i_{h_y},i_{w_y}]} = ({\sum_{i_{h_p} = 1}^{h_p} \sum_{i_{w_p} = 1}^{w_p} \bar x_{[i_{c_x},j_h,j_w]}})/{(h_p\cdot w_p)},
    \end{align}
        where $i_{c_x} \in [c_x], i_{h_y} \in [h_y], i_{w_y} \in [w_y]$,  
    \begin{align*}
        j_h & = i_{h_p}+(i_{h_y}-1)s_h-p_h,\\
        j_w &= i_{w_p}+(i_{w_y}-1)s_w-p_w,
    \end{align*}
    and 
    \begin{align*}
        \bar x_{[i_{c_x}, j_h, j_w]} = \left\{ \begin{array}{ll}
x_{[i_{c_x}, j_h, j_w]}, & \mbox{if } ( j_h \in [h_x]) \wedge (  j_w \in [w_x]),\\ 
0, & \mbox{otherwise.}
\end{array}\right.
    \end{align*}
    The height and width of the output are $h_y = \fl{\frac{h_x-h_p+p_h+s_h}{s_h}}$ and $w_y = \fl{\frac{w_x-w_p+p_w+s_w}{s_w}}$, respectively.
\end{definition}

Similar to the convolution layer, the average pooling layer can be transformed into an equivalent fully-connected layer as shown in the following lemma.

\begin{lemma}\label{lem:avg-pool}
Given an average pooling layer $\mathcal{L}_{ap}$ with parameters $\bm{\theta}_{ap} = (h_{p}, w_{p}, p_h, p_w, s_h, s_w)$, let ${\x}\in \mathbb{R}^{c_x \times h_x \times w_x}$ and ${\y}\in \mathbb{R}^{c_x \times h_y \times w_y}$ represent the input and the corresponding output of  $\mathcal{L}_{ap}$, and let $\tilde{\x}\in \mathbb{R}^{c_x \cdot h_x \cdot w_x}$ and $\tilde{\y}\in \mathbb{R}^{c_x \cdot h_y \cdot w_y}$ denote the flattened input and output vectors defined in \eqref{equ:flatten}, respectively. Then,
\begin{align}
\tilde{\y} = \tilde {\bm W}^{ap}\cdot \tilde{\x}
\end{align}
where 
$\tilde {\bm W}^{ap} =  {\bm W}^{ap}\cdot \tilde {\bm W}^p,
$ 
the matrix $\tilde {\bm W}^p \in \mathbb{R}^{(c_x \cdot (h_x+2p_h) \cdot (w_x+2p_w)) \times (c_x \cdot h_x \cdot w_x)}$ is the same as that defined in Lemma \ref{lem:conv}, and $ {\bm W}^{ap}\in \mathbb{R}^{ (c_x \cdot h_y \cdot w_y)\times (c_x \cdot (h_x+2p_h) \cdot (w_x+2p_w))}$ is a zero matrix except for the following entries:
$
{\bm W}^{ap}_{[\mathfrak j_{row},\mathfrak j_{col}]} = \frac{1}{h_p \cdot w_p}
$ 
with
\begin{align*}
\mathfrak j_{row} = & (i_{c_x}\!-\!1) \!\cdot\! h_y \!\cdot\!w_y \!+\! (i_{h_y}\!-\!1)\!\cdot\! w_y \!+\!i_{w_y},\\
\mathfrak j_{col} = & (i_{c_x}\!-\!1)\!\cdot\!(h_x\!+\!2p_h) \!\cdot\! (w_x\!+\!2p_w) \!+\! (i_{w_y}\!-\!1) \!\cdot\! s_w \!+\! i_{w_p} \\ & \!+\! (i_{h_p}\!-\!1)\!\cdot\! (w_x\!+\!2p_w) \!+\! (i_{h_y}\!-\!1)\!\cdot\! (w_x\!+\!2p_w)\!\cdot\! s_h,
\end{align*}
for any $i_{c_x}\in[{c_x}], i_{h_y} \in[{h_y}], i_{w_y}\in[{w_y}], i_{h_p}\in[{h_p}], i_{w_p}\in[{w_p}]$.
    
\end{lemma}

\subsection{Max Pooling Layer}
The max pooling layer of a CNN is defined as follows.
\begin{definition}\label{def:max-pool} 
Given an input $\x \in \mathbb{R}^{c_x \times h_x \times w_x}$ to the max pooling layer $\mathcal{L}_{mp}$ with parameters $\bm{\theta}_{mp} = (h_{p}, w_{p}, p_h, p_w, s_h, s_w)$, where $h_{p}\in \mathbb{R}$ and $w_{p}\in \mathbb{R}$ are the pooling window sizes, $p_w \in \mathbb{R}$ and $p_h \in \mathbb{R}$ are the padding sizes, and $s_w \in \mathbb{R}$ and $s_h \in \mathbb{R}$ are the strides along the width and height, respectively, the output of the max pooling layer, $\bm{y} = \mathcal{L}_{mp} (\bm{\theta}_{mp}, \x) \in \mathbb{R}^{c_x \times h_y \times w_y}$, is computed  as follows:
\begin{align}\label{equ:max-pool}
    \y_{[i_{c_x},i_{h_y},i_{w_y}]} = \max_{i_{h_p} \in [h_p], i_{w_p} \in [w_p]} \bar x_{[i_{c_x},j_h,j_w]},
\end{align}
where $i_{c_x} \in [c_x], i_{h_y} \in [h_y], i_{w_y} \in [w_y]$, 
\begin{align*}
    j_h & = i_{h_p}+(i_{h_y}-1)s_h-p_h,\\
    j_w &= i_{w_p}+(i_{w_y}-1)s_w-p_w,
\end{align*}
and 
\begin{align*}
    \bar x_{[i_{c_x}, j_h, j_w]} = \left\{ \begin{array}{ll}
x_{[i_{c_x}, j_h, j_w]}, & \mbox{if } ( j_h \in [h_x]) \wedge (  j_w \in [w_x]),\\ 
0, & \mbox{otherwise.}
\end{array}\right.
    \end{align*}
    The height and width of the output are $h_y = \fl{\frac{h_x-h_p+p_h+s_h}{s_h}}$ and $w_y = \fl{\frac{w_x-w_p+p_w+s_w}{s_w}}$, respectively.
\end{definition}

The max pooling layer can be transformed into an equivalent fully-connected layer with the maxout activation function, as shown in the following lemma.

\begin{lemma}\label{lem:max-pool}
Given a max pooling layer $\mathcal{L}_{mp}$ with parameters $\bm{\theta}_{mp} = (h_{p}, w_{p}, p_h, p_w, s_h, s_w)$, let ${\x}\in \mathbb{R}^{c_x \times h_x \times w_x}$ and ${\y}\in \mathbb{R}^{c_x \times h_y \times w_y}$ represent the input and the corresponding output of $\mathcal{L}_{mp}$, and let $\tilde{\x}\in \mathbb{R}^{c_x \cdot h_x \cdot w_x}$ and $\tilde{\y}\in \mathbb{R}^{c_x \cdot h_y \cdot w_y}$ denote the flattened input and output vectors defined in \eqref{equ:flatten}, respectively. Then,
\begin{equation}\label{equ:max-pool2}
\begin{aligned}
\tilde{\y} = [ & \max(\bm W^{mp}_{1,1,1}\cdot \tilde{\bm W}^p\cdot \tilde{\x}) \;\; \max(\bm W^{mp}_{1,1,2}\cdot \tilde{\bm W}^p\cdot \tilde{\x}) \\ & \cdots \;\; \max(\bm W^{mp}_{c_x,h_y,w_y}\cdot \tilde{\bm W}^p\cdot \tilde{\x})]^\top
\end{aligned}
\end{equation}
where $\tilde {\bm W}^p \in \mathbb{R}^{(c_x \cdot (h_x+2p_h) \cdot (w_x+2p_w)) \times (c_x \cdot h_x \cdot w_x)}$ is the same as that defined in Lemma \ref{lem:conv} and $ {\bm W}^{mp}_{i_{c_x},i_{h_y},i_{w_y}}\in \mathbb{R}^{ (h_p \cdot w_p)\times (c_x \cdot (h_x+2p_h) \cdot (w_x+2p_w))}$ is a zero matrix for $i_{c_x}\in[{c_x}], i_{h_y} \in[{h_y}], i_{w_y}\in[{w_y}]$, except for the entries at the $\mathfrak j_{row}$ rows and $\mathfrak j_{col}$ columns, which are equal to 1, i.e.,
\begin{align*}
&  ({\bm W}^{mp}_{i_{c_x},i_{h_y},i_{w_y}})_{[\mathfrak j_{row},\mathfrak j_{col}]} = 1
\end{align*}
where 
\begin{align*}
\mathfrak j_{row} = & (i_{h_p}\!-\!1) \!\cdot\! w_p \!+\!i_{w_p},\\
\mathfrak j_{col} = & (i_{c_x}\!-\!1)\!\cdot\!(h_x\!+\!2p_h) \!\cdot\! (w_x\!+\!2p_w) \!+\! (i_{w_y}\!-\!1) \!\cdot\! s_w \!+\! i_{w_p} \\ & \!+\! (i_{h_p}\!-\!1)\!\cdot\! (w_x\!+\!2p_w) \!+\! (i_{h_y}\!-\!1)\!\cdot\! (w_x\!+\!2p_w)\!\cdot\! s_h,
\end{align*}
for any $i_{h_p}\in[{h_p}], i_{w_p}\in[{w_p}]$.
\end{lemma}


Denote $\mathcal{L}_{mp\rightarrow fc}: \mathbb{R}^{c_x \cdot h_x \cdot w_x} \rightarrow  \mathbb{R}^{c_x \cdot h_y \cdot w_y} $ as the fully-connected layer transformed from a max pooling layer in Lemma \ref{lem:max-pool}. Given an input set $\mathcal{X}\subset \mathbb{R}^{c_x \cdot h_x \cdot w_x}$ in the form of an HZ, the output set of $\mathcal{L}_{mp\rightarrow fc}$ can be also represented as an HZ.
Indeed, using \eqref{equ:max-pool2}, the output set of $\mathcal{L}_{mp\rightarrow fc}$ can be constructed as 
\begin{align*}
    \mathcal{Y} = & [  \max(\bm W^{mp}_{1,1,1}\cdot \tilde{\bm W}^p\cdot \mathcal{X}) \;\; \max(\bm W^{mp}_{1,1,2}\cdot \tilde{\bm W}^p\cdot \mathcal{X}) \\ & \cdots \;\; \max(\bm W^{mp}_{c_x,h_y,w_y}\cdot \tilde{\bm W}^p\cdot\mathcal{X})]^\top\\ 
    = &  \max(\bm W^{mp}_{1,1,1}\cdot \tilde{\bm W}^p\cdot \mathcal{X}) \times \max(\bm W^{mp}_{1,1,2}\cdot \tilde{\bm W}^p\cdot \mathcal{X})\times \\ &  \cdots \times \max(\bm W^{mp}_{c_x,h_y,w_y}\cdot \tilde{\bm W}^p\cdot\mathcal{X}).
\end{align*}
Since HZs are closed under linear mapping \cite{bird2023hybrid}, we know $(\bm W^{mp}_{i_{c_x},i_{h_y},i_{w_y}}\cdot \tilde{\bm W}^p\cdot \mathcal{X})$ is an HZ for all $i_{c_x}\in[{c_x}], i_{h_y} \in[{h_y}], i_{w_y}\in[{w_y}]$. As each maxout activation function can be viewed as a piecewise linear function \cite{goodfellow2013maxout}, and a piecewise linear function can be exactly represented as an HZ \cite{zhang2024hybrid}, $\mathcal{X}^{max}_{i_{c_x},i_{h_y},i_{w_y}} \triangleq \max(\bm W^{mp}_{i_{c_x},i_{h_y},i_{w_y}}\cdot \tilde{\bm W}^p\cdot \mathcal{X})$ is also an HZ for all $i_{c_x}\in[{c_x}], i_{h_y} \in[{h_y}], i_{w_y}\in[{w_y}]$. Therefore, $\mathcal{Y}$ can be computed as the Cartesian product of $c_x \cdot h_y \cdot w_y$ HZs, which is also an HZ.

For a given CNN $\mt{\pi}_C$, the convolution and pooling layers can be transformed into corresponding fully-connected layers by using Lemma \ref{lem:conv}, Lemma \ref{lem:avg-pool} and Lemma \ref{lem:max-pool}. Other layers, such as batch normalization, can be converted in a similar manner. As a result, we can construct an  FFNN ${\mt{\pi}}_F$ that is equivalent to $\mt{\pi}_C$ in the sense that $\mt{\pi}_C$ and ${\mt{\pi}}_F$ have the same input-output mapping. In the next section, we will compute tight over-approximations of the reachable sets for  $\mt{\pi}_C$ based on ${\mt{\pi}}_F$.
    

\section{Reachability Analysis of CNN Using HZ}\label{sec:FFNN}
In this section, we propose a novel approach to efficiently reduce an FFNN while preserving its input-output relation with a  tunable bounded error. This is followed by the application of an HZ-based reachable set computation method for the reduced network.


\subsection{A Novel Reduction Approach For FFNNs }
In \cite{zhang2023reachability}, it was shown  that the input-output relationship of a ReLU-activated FFNN can be exactly expressed as an HZ. Consequently, the reachable set of an FFNN can be represented in closed form by an  HZ for a given HZ-represented input set.  However, the complexity of the HZ representation for the output increases proportionally with the total number of neurons in the FFNN. To improve the scalability of the HZ-based method, an effective method for compressing the FFNN - reducing the number of neurons while preserving its input-output mapping  - is essential. 



Consider an $\ell$-layer FFNN $\mt{\pi}_F$ with weight matrices $\{\bm W^{(k)}\}_{k\in[\ell]}$ and bias vectors $\{\bm v^{(k)}\}_{k\in[\ell]}$. Let an HZ $\mathcal{Z}$ be the input set of $\mt{\pi}_F$. Denote $\mathcal{J}^{(k)} \triangleq \li \bm \alpha^{(k)}, \bm \beta^{(k)} \ri \subset \mathbb{R}^{n_k}$ as the interval bound  of the ranges of neurons in the $k$-th layer of $\mt{\pi}_F$, i.e., 
\begin{align*}
    \mathcal{J}^{(k)} \supseteq \{ & \x^{(k)} \in \mathbb{R}^{n_k} | \x^{(i)} = \mathcal{L}_{fc}(\bm W^{(i)},\bm v^{(i)}, \bm x^{(i-1)}),  i\in [k], \\ &\quad\quad\quad\quad\quad \x^{(0)}\in \mathcal{Z} \}, \text{ for } k\in [\ell-1],\\
    \mathcal{J}^{(\ell)} \supseteq \{ & \mt{\pi}_F(\x) \in \mathbb{R}^{m} | \x\in \mathcal{Z} \}.
\end{align*}
The following lemma from \cite[Proposition 4]{ladner2023fully} shows that a reduced FFNN, $\hat{\mt{\pi}}_F$, with fewer neurons can be constructed by adjusting the weights and bias of the original FFNN, $\mt{\pi}_F$, such that $\hat{\mt{\pi}}_F$ over-approximates ${\mt{\pi}}_F$ in $\mathcal{Z}$.

\begin{lemma}\label{lem1} 
For the $k$-th layer of an FFNN $\mt{\pi}_F$, $k\in[\ell-1]$, given the interval bounds $\mathcal{J}^{(k)} \subset \mathbb{R}^{n_{k}}$ for the neurons in the $k$-th layer and the index set $\mathcal{N}^{(k)}\subseteq [n_{k}]$, 
a reduced network $\hat{\mt{\pi}}_F$ is constructed by adjusting the weights and bias of the $k$-th and $(k+1)$-th layers as follows:
\begin{equation}\label{equ:WBtilde}
\begin{aligned}
& \hat{\bm W}^{(k)}\!=\!\bm W^{(k)}_{[\overline{\mathcal{N}}^{(k)},:]},\; \hat{\bm b}^{(k)}=\bm b^{(k)}_{[\overline{\mathcal{N}}^{(k)},:]}, \\
& \hat{\bm W}^{(k+1)}\!=\!\bm W^{(k+1)}_{[:,\overline{\mathcal{N}}^{(k)}]},\; \hat{\bm b}^{(k+1)}=\bm b^{(k+1)}+\bm{\varepsilon}^{(k+1)},
\end{aligned}
\end{equation}
where $\overline{\mathcal{N}}^{(k)}\triangleq\left[n_{k}\right] \backslash \mathcal{N}^{(k)}$ denotes the index set of remaining neurons and $\hat{\bm b}^{(k+1)}$ includes the approximation error 
\begin{align}\label{equ:error}
    \bm{\varepsilon}^{(k+1)} \triangleq \bm W^{(k+1)}_{[:,{\mathcal{N}}^{(k)}]}\cdot \proj_{{\mathcal{N}}^{(k)}}(\mathcal{J}^{(k)}).
\end{align}
Then, 
$\mt{\pi}_F(\x) \in \hat{\mt{\pi}}_F(\x)$, $\forall \x\in \mathcal{Z}$.
    \end{lemma}

\begin{algorithm}[!hbt]
\SetNoFillComment
\caption{Determination of Reduced Neurons}\label{alg:identify}
\KwIn{Weight matrix $\bm W^{(k+1)} \in \mathbb{R}^{n_{k+1}\times n_k}$, interval bounds $\mathcal{J}^{(k)} = \li \bm \alpha^{(k)}, \bm \beta^{(k)} \ri  \subset \mathbb{R}^{n_{k}}$, error bound $\rho \in \mathbb{R}_{\geq 0}$}
\KwOut{Index set of reduced neurons in the $k$-th layer $\mathcal{N}^{(k)}$}
$\mathcal{N}^{(k)}$ $\leftarrow$ $\emptyset$;\\
$\bm d^{(k)}$  $\leftarrow$ $\bm\beta^{(k)} - \bm\alpha^{(k)}$ \tcp*{Interval diameter}
$\vec{\bm W}^{(k+1)}$ $\leftarrow$ $\sum_{i=1}^{n_{k+1}} |\bm W^{(k+1)}_{[i,:]}|$ \tcp*{Column-wise sums of the absolute values}
${\bm h}^{(k+1)}$ $\leftarrow$ $\vec{\bm W}^{(k+1)}\odot \bm d^{(k)}$ \tcp*{Hadamard product}
\For{$j \in \{1,2,\dots,n_k$\}}{
\If{$h^{(k+1)}_j \leq \rho$}
{$\mathcal{N}^{(k)}$ $\leftarrow$ $\mathcal{N}^{(k)}\cup \{j\}$\tcp*{Append index $j$}}
}
\Return{$\mathcal{N}^{(k)}$}
\end{algorithm}

The main idea of our FFNN reduction approach is straightforward: in the $k$-th layer, we eliminate neurons that contribute minimally to the $(k+1)$-th layer and introduce a compensatory term to account for the approximation error caused by this elimination in the $(k+1)$-th layer, such that the reduced FFNN over-approximates the original FFNN. Using Lemma \ref{lem1}, the  number of neurons of the $k$-th layer is reduced from $n_k$ to $n_k - |\mathcal{N}^{(k)}|$ and the resulting approximation error is $\bm{\varepsilon}^{(k+1)} = \li \underline{\bm{\varepsilon}}^{(k+1)}, \overline{\bm{\varepsilon}} ^{(k+1)}\ri\subset \mathbb{R}^{n_{k+1}}$ represented as an interval. To minimize  error propagation through subsequent layers of the reduced FFNN, we constrain the size of the error $\bm{\varepsilon}^{(k+1)}$ defined as $size(\bm{\varepsilon}^{(k+1)}) \triangleq \sum_{i=1}^{n_{k+1}} (\overline{\varepsilon}^{(k+1)}_i - \underline{\varepsilon}^{(k+1)}_i)$. Algorithm \ref{alg:identify} summarizes the procedure for identifying the maximum number of neurons that can be reduced while controlling the size of the approximation error for each layer. 

The following proposition shows that the size of the approximation error is minimized.

\begin{proposition}\label{prop:reduce}
Given the weight matrix $\bm W^{(k+1)} \in \mathbb{R}^{n_{k+1}\times n_k}$ for the $(k+1)$-th layer, the interval bounds $\mathcal{J}^{(k)} = \li \bm \alpha^{(k)}, \bm \beta^{(k)} \ri\subset \mathbb{R}^{n_{k}}$ for  neurons in the $k$-th layer, and a scalar $\rho \geq 0$, let $\mathcal{N}^{(k)}$ be the index set of reduced neurons determined by Algorithm \ref{alg:identify} and $\bm{\varepsilon}^{(k+1)}$ the induced approximation error as defined in \eqref{equ:error}. Then, with the same number of reduced neurons, the size of the error, $size(\bm{\varepsilon}^{(k+1)})$, is minimized and satisfies $size(\bm{\varepsilon}^{(k+1)}) \leq \rho\cdot|\mathcal{N}^{(k)}|$.
\end{proposition}


\begin{proof}
Using the  approximation error  defined in \eqref{equ:error} and the interval arithmetic, we have 
\begin{align*}
    size(\bm{\varepsilon}^{(k+1)}) = & size(\bm W^{(k+1)}_{[:,{\mathcal{N}}^{(k)}]}\cdot \proj_{{\mathcal{N}}^{(k)}}(\mathcal{J}^{(k)}))\\
     = & size(\sum_{j\in \mathcal{N}^{(k)}} \bm W^{(k+1)}_{[:,j]} \cdot \mathcal{J}^{(k)}_j)\\
     = & \sum_{i =1}^{n_{k+1}}\sum_{j\in \mathcal{N}^{(k)}} |\bm W^{(k+1)}_{[i,j]}| \cdot ( \beta^{(k)}_j - \alpha^{(k)}_j)\\
     = & \sum_{j\in \mathcal{N}^{(k)}} (\sum_{i =1}^{n_{k+1}} |\bm W^{(k+1)}_{[i,j]}| \cdot ( \beta^{(k)}_j - \alpha^{(k)}_j))\\
     = & \sum_{j\in \mathcal{N}^{(k)}} \vec{\bm W}^{(k+1)}_j \cdot d^{(k)}_j \\
     = & \sum_{j\in \mathcal{N}^{(k)}} \vec{\bm W}^{(k+1)}_j\odot \bm d^{(k)}_j \\
     = & \sum_{j\in \mathcal{N}^{(k)}} h_j^{(k+1)} \leq \rho \cdot |\mathcal{N}^{(k)}|
\end{align*}
where the last inequality is from $h_j^{(k+1)} \leq \rho,\forall j \in \mathcal{N}^{(k)}$. Since $\bar j\in \overline{\mathcal{N}}^{(k)}$ implies  $h_{\bar j}^{(k+1)} = \sum_{i =1}^{n_{k+1}} |\bm W^{(k+1)}_{[i,{\bar j}]}| \cdot ( \beta^{(k)}_{\bar j} - \alpha^{(k)}_{\bar j}) > \rho$, the value of $size(\bm{\varepsilon}^{(k+1)})$ is minimized with the same number of reduced neurons.
\end{proof}

\subsection{Over-approximated Reachable Sets for FFNNs and CNNs}

By applying Algorithm \ref{alg:identify} and Proposition \ref{prop:reduce} to each layer of FFNN $\mt{\pi}_F$, a reduced FFNN $\hat{\mt{\pi}}_F$ can be constructed with weight matrices $\{\hat{\bm W}\}_{k\in [\ell]}$ and bias $\{\hat{\bm v}\}_{k\in [\ell]}$. To compute the reachable set of $\hat{\mt{\pi}}_F$, we will use the following lemma from \cite[Proposition 2]{zhang2024reachability}, which states that the graph of the activation function $\bm{\sigma}$ can be over-approximated by an HZ.

\begin{lemma}\label{lem:sigma} 
Given an HZ $\mathcal{Z}\subset \mathbb{R}^{n_{k}}$, its interval hull $\mathcal{J}\triangleq \li\mt{\alpha},\mt{\beta}\ri = \itl(\mathcal{Z})$ and a tunable relaxation parameter $0\leq\gamma\leq 1$, the graph of the activation function $\mt\sigma: \mathbb{R}^{n_{k}}\rightarrow \mathbb{R}^{n_{k}}$ over $\mathcal{Z}$ can be over-approximated by the HZ: 
     \begin{equation}\label{equ:g_phi_tilde}
         \hat{\mathcal{G}}_{\mt\sigma}(\mathcal{Z}) = (\bm{P}\cdot \hat{\mathcal{G}}_{ReLU}(\mathcal{J}))\cap_{[\bm{I}\; \bm{0}]} \mathcal{Z} \supseteq \mathcal{G}_{\mt\sigma}(\mathcal{Z})
     \end{equation}
    where $\bm{P}=[\bm{e}_{2}\; \bm{e}_{4}\;\cdots\; \bm{e}_{2n_{k}} \; \bm{e}_{1}\; \bm{e}_{3}\;\cdots\; \bm{e}_{2n_{k}-1}]^T\in \mathbb{R}^{2n_{k}\times 2n_{k}}$ is a permutation matrix 
    , $\hat{\mathcal{G}}_{ReLU}(\mathcal{J}) = \hat{\mathcal{G}}_{ReLU}(\li\alpha_1,\beta_1 \ri) \times \cdots \times \hat{\mathcal{G}}_{ReLU}(\li\alpha_{n_k},\beta_{n_k}\ri)$ and 
    \begin{align*}
&\hat{\mathcal{G}}_{ReLU}(\li\alpha_i,\beta_i\ri) =\nonumber\\
&\left\{ \begin{array}{l}
\mathcal{H}_+\triangleq \lrangle{\mat{\frac{\beta_i-\alpha_i}{2} \\ \frac{\beta_i-\alpha_i}{2}},\emptyset,\mat{\frac{\beta_i+\alpha_i}{2} \\ \frac{\beta_i+\alpha_i}{2}},\emptyset,\emptyset,\emptyset}, \;\; \mbox{if} \; 0\leq  \alpha_i,\\ 
\mathcal{H}_-\triangleq \lrangle{\mat{\frac{\beta_i-\alpha_i}{2} \\ 0}, \emptyset,\mat{\frac{\beta_i+\alpha_i}{2} \\ 0}, \emptyset,\emptyset,\emptyset}, \;\; \mbox{if}\; \beta_i \leq 0,\\
\mathcal{H}_\pm\triangleq \mathcal{H}_+ \cup \mathcal{H}_-, \;\; \mbox{if}\; \alpha_i \!<\!0\!<\!\beta_i \wedge \frac{|\alpha_i|}{\beta_i} \!>\! \gamma \wedge \frac{\beta_i}{|\alpha_i|} \!>\! \gamma,\\
\mathcal{C}_\triangle \triangleq {ConvexHull}(\mathcal{H}_\pm), \;\; \mbox{otherwise},
\end{array}\right.
\end{align*} 
for $i\in[n_k]$. 
Moreover, when $\gamma = 0$, $\hat{\mathcal{G}}_{\mt\sigma}(\mathcal{Z}) = \mathcal{G}_{\mt\sigma}(\mathcal{Z})$. 
\end{lemma}

\begin{algorithm}[!ht]
\SetNoFillComment
\caption{Over-approximated Reachability Analysis of FFNN Using HZ}\label{alg:reachFFNN}
\KwIn{HZ input set $\mathcal{Z}$, original FFNN $\mt{\pi}_F$ with weight matrices $\{\bm W^{(k)}\}_{k=1}^{\ell}$ and bias vectors $\{\bm v^{(k)}\}_{k=1}^{\ell}$, error bound $\rho \geq 0$, HZ relaxation parameter $0\leq\gamma\leq 1$}
\KwOut{Over-approximated reachable set $\hat{\mathcal{R}}_{\mt{\pi}_F}$ as an HZ}
$\mathcal{X}^{(0)}$ $\leftarrow$ $\mathcal{Z} = \hz{z}$;\\
$\hat{\bm W}^{(1)}$ $\leftarrow$ $ \bm W^{(1)}$; $\hat{\bm v}^{(1)}$ $\leftarrow$ $ \bm v^{(1)}$;\\
\For{$k \in \{1,2,\dots,\ell-1$\}}{
$\mathcal{Z}^{(k)} \leftarrow \hat{\bm{W}}^{(k)}\mathcal{X}^{(k-1)}\!+\!\hat{\bm{v}}^{(k)}$; \\
$ \mathcal{J}^{(k)}_z = \li \mt{\alpha}^{(k)}, \mt{\beta}^{(k)}\ri \leftarrow CROWN(\mathcal{Z}^{(k)})$;\\
$ \mathcal{J}^{(k)}_x\!\! \leftarrow \!\mt{\sigma}(\mathcal{J}^{(k)}_z)$;\\
$\mathcal{N}^{(k)} \leftarrow  $ Algorithm \ref{alg:identify} with $\bm{W}^{(k+1)}$, $\mathcal{J}^{(k)}_x$ and $\rho$;\\
$\hat{\bm W}^{(k)},\hat{\bm v}^{(k)},\hat{\bm W}^{(k+1)},\hat{\bm v}^{(k+1)} \leftarrow $ \eqref{equ:WBtilde} in Lemma \ref{lem1};\\
$\hat{\mathcal{Z}}^{(k)} \leftarrow \proj_{\overline{\mathcal{N}}^{(k)}} (\mathcal{Z}^{(k)})$;\\
$\hat{\mathcal{J}}^{(k)}_z \leftarrow \proj_{\overline{\mathcal{N}}^{(k)}} (\mathcal{J}^{(k)}_z)$;\\
$\hat{\mathcal{G}}^{(k)}$ $\leftarrow$ $(\bm{P}\cdot \hat{\mathcal{G}}_{ReLU}(\hat{\mathcal{J}}^{(k)}_z))\cap_{[\bm{I}\; \bm{0}]} \hat{\mathcal{Z}}^{(k)}$;\\
$\mathcal{X}^{(k)}$ $\leftarrow$ $[\bm{0}\;\bm{I}] \cdot \hat{\mathcal{G}}^{(k)}$;\\
}
$\hat{\mathcal{R}}_{\mt{\pi}_F}$ $\leftarrow$ $\hat{\bm{W}}^{(\ell)}\mathcal{X}^{(\ell-1)}+\hat{\bm{v}}^{(\ell)}$;\\
\Return{$\hat{\mathcal{R}}_{\mt{\pi}_F}$}
\end{algorithm}

To construct an HZ over-approximation of the reachable set for $\mt\pi_F$, we propagate the input set through the reduced FNN $\hat{\mt\pi}_F$ layer by layer using Lemma \ref{lem:sigma} and linear map operations on HZs. The detailed procedure is  summarized in Algorithm \ref{alg:reachFFNN}, which is adapted from \cite[Algorithm 2]{zhang2024reachability}. Notably, in Line 5, we use CROWN \cite{zhang2018efficient,wang2021betaetal} with GPU-accelerated computation to efficiently approximate the interval bounds for each layer in the reduced FFNN. Compared with our previous method in \cite{zhang2024reachability}, which computes an exact interval hull of an HZ by solving a set of MILPs for each layer,  using CROWN significantly enhances the scalability of the proposed method without substantially impacting the accuracy of the HZ-represented reachable set.

The following theorem shows that the reachable set over-approximation in Algorithm \ref{alg:reachFFNN} is sound.

\begin{theorem}\label{thm:alg-reduce}
Given an $\ell$-layer ReLU-activated FNN $\mt{\pi}_F: \mathbb{R}^n\rightarrow \mathbb{R}^m$ and an HZ $\mathcal{Z}\subset\mathbb{R}^n$, the output of Algorithm \ref{alg:reachFFNN}, $\hat{\mathcal{R}}_{\mt{\pi}_F}$, is an HZ that over-approximates the exact reachable set of $\mt{\pi}_F$ over $\mathcal{Z}$, i.e., $\hat{\mathcal{R}}_{\mt{\pi}_F} \supseteq {\mathcal{R}}_{\mt{\pi}_F}(\mathcal{Z})$. Furthermore, $\hat{\mathcal{R}}_{\mt{\pi}_F} ={\mathcal{R}}_{\mt{\pi}_F}(\mathcal{Z})$ when $\rho = 0$ and $\gamma = 0$.
\end{theorem}
\begin{proof}
A reduced FFNN $\hat{\mt{\pi}}_F$ is constructed using  \eqref{equ:WBtilde} in Line 8 of Algorithm \ref{alg:reachFFNN}. In Line 9-12, the over-approximated input set $\hat{\mathcal{Z}}^{(k)}$, graph $\hat{\mathcal{G}}^{(k)}$ and  output set $\mathcal{X}^{(k)}$ of the $k$-th layer of the reduced FFNN is computed iteratively for $k\in[\ell-1]$. Thus, the over-approximation properties are preserved through the propagation of each hidden layer. Only a linear map is applied to the last layer in Line 13. As $\hat{\mt\pi}_F$ over-approximates $\mt{\pi}_F$ over  $\mathcal{Z}$, $\hat{\mathcal{R}}_{\mt{\pi}_F}\supseteq \mathcal{R}_{\hat{\mt\pi}_F}(\mathcal{Z})\supseteq \mathcal{R}_{{\mt\pi}_F}(\mathcal{Z})$. Since $\mathcal{Z}$ is an HZ and HZs are closed under all the set operations involved in Algorithm \ref{alg:reachFFNN}, $\hat{\mathcal{R}}_{\mt{\pi}_F}$ is also an HZ by construction. 
When $\rho = \gamma = 0$, $\hat{\mt\pi}_F$ preserves the same input-output relationship of $\mt\pi_F$ as shown in Lemma \ref{lem1}, and the constructed graph set in Line 11 is exact for each layer by Lemma \ref{lem:sigma}. Thus, $\hat{\mathcal{R}}_{\mt{\pi}_F} ={\mathcal{R}}_{\mt{\pi}_F}(\mathcal{Z})$. 
\end{proof}


Finally, for a given CNN $\mt{\pi}_C$, its over-approximated reachable set can be computed by directly applying Theorem \ref{thm:alg-reduce} to the FFNN ${\mt{\pi}}_F$ that is constructed in Section \ref{subsec:CNN} and equivalent to $\mt{\pi}_C$, as shown in the following result.


\begin{corollary}\label{cor:CNN}
    Given a CNN $\mt{\pi}_C:\mathbb{R}^{c_I \times h_I\times w_I} \rightarrow \mathbb{R}^m$ with an input set $\mathcal{I} \subset \mathbb{R}^{c_I \times h_I\times w_I}$, let ${\mt{\pi}}_F: \mathbb{R}^{c_I \cdot h_I \cdot w_I} \rightarrow \mathbb{R}^m$ be the FFNN transformed from $\mt{\pi}_C$ using Lemma \ref{lem:conv} - Lemma \ref{lem:max-pool}, and denote $\hat{\mathcal{R}}_{{\mt{\pi}}_F}$ as the result by applying Algorithm \ref{alg:reachFFNN} to ${\mt{\pi}}_F$, then $\hat{\mathcal{R}}_{{\mt{\pi}}_F} \supseteq \mathcal{R}_{\mt{\pi}_C}(\mathcal{I})$. Moreover, when $\rho = \gamma = 0$ in Algorithm \ref{alg:reachFFNN}, $\hat{\mathcal{R}}_{{\mt{\pi}}_F} = \mathcal{R}_{\mt{\pi}_C}(\mathcal{I})$.
\end{corollary}

\section{Simulation Results}\label{sec:example}
Two simulation examples are provided to demonstrate the effectiveness of the proposed method. The implementation was carried out in Python and executed on a desktop with an Intel Core i5-11400F CPU and 32GB of RAM. 

\subsection{Robustness Verification of MNIST CNN}
In the first example, we use two CNNs from the ERAN benchmark in \cite{bak2021second}, 
trained on the MNIST dataset, which consists of 60,000 images of handwritten digits with a resolution of $1 \times 28 \times 28$ pixels. Each pixel has a value from 1 to 255. Based on the architectures of the CNNs, they are referred to as the small and large MNIST CNNs: $\mt{\pi}_{C,\text{small}}, \mt{\pi}_{C,\text{large}}: \mathbb{R}^{1 \times 28 \times 28} \rightarrow \mathbb{R}^{10}$. 
The CNNs classify the digit images into ten classes: $0,1,\dots,9$, and the classified output corresponds to the index of the maximum value among the 10 output dimensions.

\begin{figure}[!t]
    \centering
    \begin{subfigure}[b]{0.43\linewidth}
        \centering
        \includegraphics[width=0.95\linewidth]{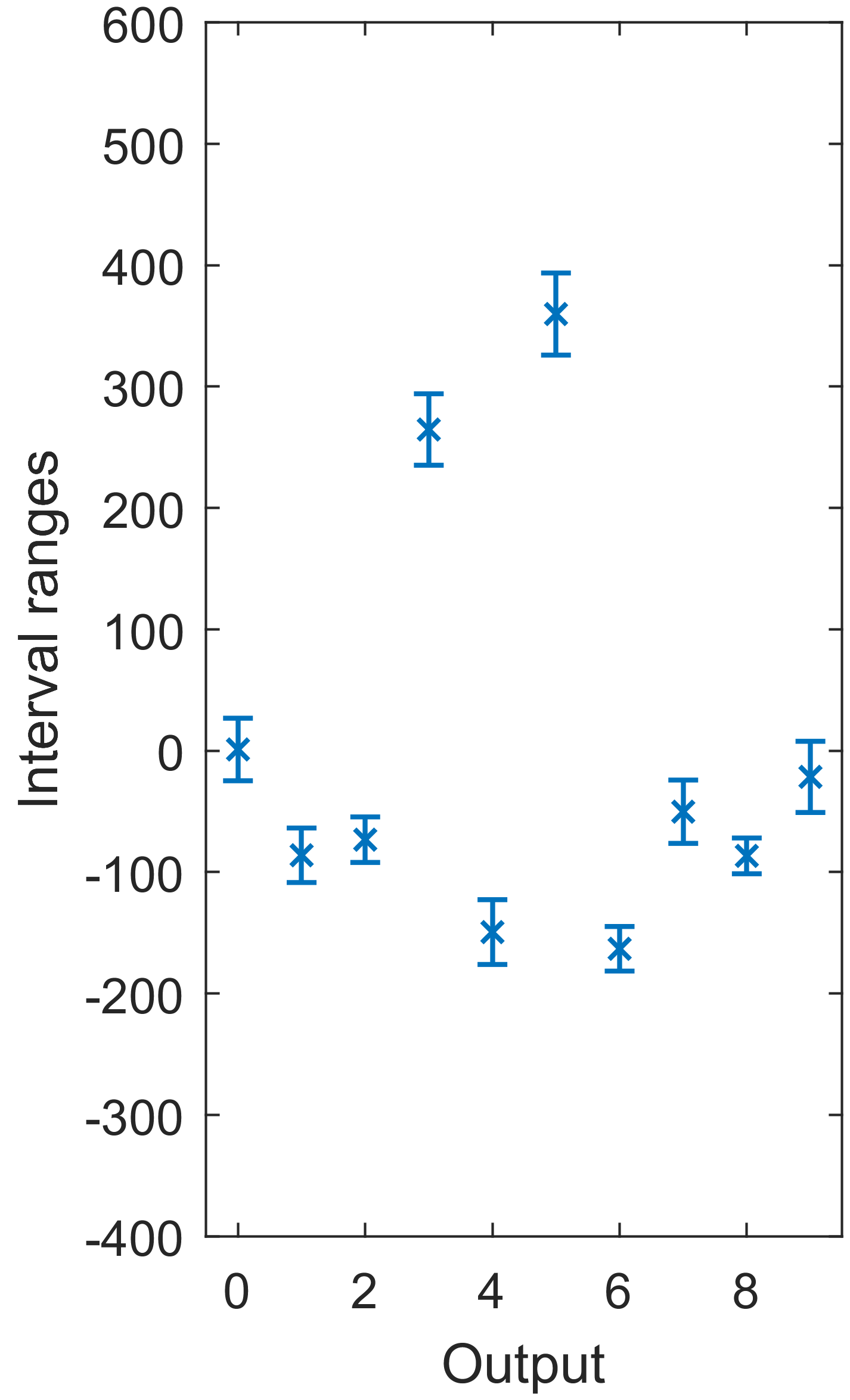}
        \caption{This work}
    \end{subfigure}
    \hspace{2em}
    \begin{subfigure}[b]{0.43\linewidth}
        \centering
        \includegraphics[width=0.95\linewidth]{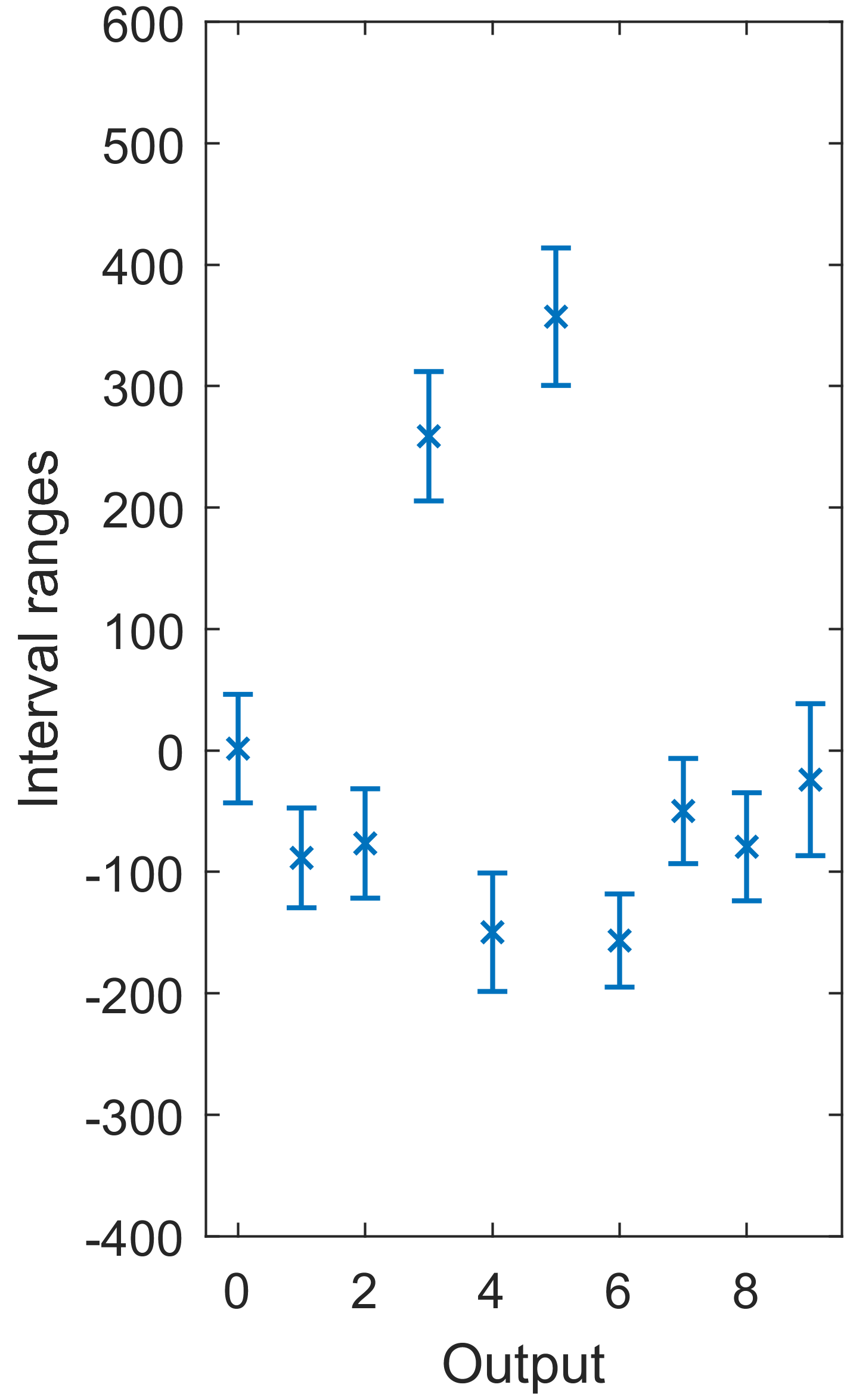}
        \caption{$\alpha$-CROWN \cite{Xu2021fast}}
    \end{subfigure}
    \caption{An example of output ranges of the small MNIST CNN with $d = 200$ and $\delta = 0.05$.}
    \label{fig:ranges}
\end{figure}

\begin{table}[!t]
\centering
\begin{tabular}{|c|cc|cc|}
\hline
\multirow{2}{*}{} & \multicolumn{2}{c|}{Robustness results (\%)} & \multicolumn{2}{c|}{Average runtimes (s)}               \\ \cline{2-5} 
                             & \multicolumn{1}{c|}{Ours}& $\alpha$-CROWN  & \multicolumn{1}{c|}{Ours}   & $\alpha$-CROWN      \\ \hline
\makecell{$d \!=\! 245$ \\ $\delta \!=\! 0.01$}  & \multicolumn{1}{c|}{86.2}  & 86.2 & \multicolumn{1}{c|}{0.514} & 0.259 \\ \hline
\makecell{$d \!=\! 230$ \\ $\delta \!=\! 0.015$} & \multicolumn{1}{c|}{80.9}     & 80.9  & \multicolumn{1}{c|}{0.706} & 0.343 \\ \hline
\makecell{$d \!=\! 200$ \\ $\delta \!=\! 0.05$} & \multicolumn{1}{c|}{69.2}      & 67.9  & \multicolumn{1}{c|}{1.751}  & 0.592 \\ \hline
\end{tabular}
\caption{Robustness verification results for the small MNIST CNN on 1000 randomly perturbed image input.}\label{tab:1}
\end{table}

\begin{table}[!t]
\centering
\begin{tabular}{|c|cc|cc|}
\hline
\multirow{2}{*}{} & \multicolumn{2}{c|}{Robustness results (\%)} & \multicolumn{2}{c|}{Average runtimes (s)}               \\ \cline{2-5} 
                             & \multicolumn{1}{c|}{Ours}  & $\alpha$-CROWN  & \multicolumn{1}{c|}{Ours} & $\alpha$-CROWN      \\ \hline
\makecell{$d \!=\! 245$ \\ $\delta \!=\! 0.01$}  & \multicolumn{1}{c|}{86.5} & 86.5 & \multicolumn{1}{c|}{0.809} & 0.319 \\ \hline
\makecell{$d \!=\! 230$ \\ $\delta \!=\! 0.015$} & \multicolumn{1}{c|}{80.5}      & 80.4  & \multicolumn{1}{c|}{1.420}   & 0.412 \\ \hline
\makecell{$d \!=\! 200$ \\ $\delta \!=\! 0.05$} & \multicolumn{1}{c|}{68.8}     &  66.7 & \multicolumn{1}{c|}{3.604}  & 0.655 \\ \hline
\end{tabular}
\caption{Robustness verification results for the large MNIST CNN on 1000 randomly perturbed image input.}\label{tab:2}
\end{table}

To evaluate the robustness of the MNIST CNNs, we apply a brightening attack \cite{gehr2018ai2}, which darkens parts of the input images. Similar to the approach in \cite{tran2020verification}, the perturbed input images are constructed by adjusting pixels with values greater than a threshold $d$ into an interval range $\li 0, 255 \times \delta \ri$, where $\delta \geq 0$ determines the size of the perturbation. Mathematically, given an image $I \in \mathbb{R}^{1\times 28 \times 28}$, the perturbed input image is represented as an interval $\mathcal{I}\triangleq \li \underline{I}, \overline{I} \ri \subset \mathbb{R}^{1\times 28 \times 28}$, where $\underline{I}_{[1,i,j]} = 0$ and $\overline{I}_{[1,i,j]} = 255 \times \delta$ if $I_{[1,i,j]} \geq d$, and $\underline{I}_{[1,i,j]} = \overline{I}_{[1,i,j]} = I_{[1,i,j]}$ otherwise, for $i,j = 1,2,\dots,28$.

We randomly select 1000 images from the MNIST dataset and construct 1000 corresponding perturbed input sets based on the threshold $d$ and perturbation radius $\delta$. Algorithm \ref{alg:reachFFNN} and Corollary \ref{cor:CNN} are then applied to compute the reachable sets for both CNNs. A CNN is considered robust against the brightening attack if the correctly classified output consistently maintains the maximum value across all other outputs within the computed reachable set. 
For comparison, we also use $\alpha$-CROWN \cite{Xu2021fast} with default settings to verify the robustness of the two MNIST CNNs. The success rates for robustness verification and the average computation times for each method are summarized in Table \ref{tab:1} and Table \ref{tab:2}. 
An example of the output ranges for the small MNIST CNN is illustrated in Figure \ref{fig:ranges}. The results show that the proposed method is less conservative and achieves an equal or higher verification success rate with comparable computation time. Notably, $\alpha$-CROWN represents CNN reachable sets as \emph{intervals}, whereas  our proposed method represent them as \emph{HZs}, a more general and expressive set representation. Although  $\alpha$-CROWN achieves shorter computation times, our HZ-represented reachable sets offer greater accuracy than interval-based reachability by yielding \emph{smaller set volumes}. This improvement enhances verification accuracy, particularly in the presence of large disturbances.


\subsection{Robustness Verification of CIFAR-10 CNN}
In the second example, we use a CNN trained on the CIFAR-10 colored image dataset from the benchmark in \cite{bak2021second}. The CNN consists of four
convolutional layers with up to 32768 neurons per layer and two fully-connected layers. Similar to the MNIST CNNs, the CIFAR-10 CNN classifies the input image $I \in \mathbb{R}^{3\times 32 \times 32}$ into ten categories. We assume each pixel of the image inputs is perturbed by a radius $\delta \in \{0.001,0.01\}$. Using the method developed in Section \ref{subsec:CNN}, the CIFAR-10 CNN is transformed into an equivalent FFNN. To enhance the efficiency of reachable set computation, Algorithm \ref{alg:identify} is employed to determine the neurons to be reduced during the reachable set computation in Algorithm \ref{alg:reachFFNN}.  Table \ref{tab:3} summarizes the number of neurons in each layer of the original CIFAR-10 CNN and the reduced network using Algorithm \ref{alg:identify}. Our neuron reduction approach achieves a 90 \% reduction in each layer of the transformed FFNN, significantly simplifying the subsequent reachability analysis. The average computation time for neuron reduction and reachable set calculation is 3.19 seconds.

\begin{table}[!t]
\centering
\begin{tabular}{|c|cccccc|}
\hline
                                  & \multicolumn{6}{c|}{Number of neurons}                                                                                                           \\ \hline
Layer                             & \multicolumn{1}{c|}{1}     & \multicolumn{1}{c|}{2}    & \multicolumn{1}{c|}{3}     & \multicolumn{1}{c|}{4}    & \multicolumn{1}{c|}{5}   & 6   \\ \hline
Original CNN                      & \multicolumn{1}{c|}{32768} & \multicolumn{1}{c|}{8192} & \multicolumn{1}{c|}{16384} & \multicolumn{1}{c|}{4096} & \multicolumn{1}{c|}{512} & 512 \\ \hline
\makecell{Reduced CNN \\ with $\delta = 0.001$} & \multicolumn{1}{c|}{2613}  & \multicolumn{1}{c|}{251}  & \multicolumn{1}{c|}{115}   & \multicolumn{1}{c|}{176}  & \multicolumn{1}{c|}{56}  & 84  \\ \hline
\makecell{Reduced CNN \\ with $\delta = 0.01$ } & \multicolumn{1}{c|}{2752}      & \multicolumn{1}{c|}{513}     & \multicolumn{1}{c|}{114}      & \multicolumn{1}{c|}{113}     & \multicolumn{1}{c|}{51}    &   269  \\ \hline
\end{tabular}
\caption{Neuron reduction results for CIFAR-10 CNN.}\label{tab:3}
\end{table}

\section{Conclusion}\label{sec:concl}
In this work, we proposed an efficient HZ-based approach for computing the reachable sets of CNNs. By expressing the convolution and pooling operations as linear mappings, we demonstrated that a CNN can be transformed into an equivalent FFNN. To enhance the efficiency of reachability analysis for the converted FFNN, we developed a flexible neural network reduction method that allows for closed-form over-approximations of the CNN's reachable sets in the form of HZs. 
The performance of the proposed approach was evaluated using two numerical examples.

\bibliographystyle{IEEEtran}
\bibliography{ref}

\end{document}